\documentclass{amsart}

\calclayout

\usepackage{amsmath}
\usepackage{amssymb}
\usepackage[all]{xy}
\usepackage{enumerate}
\usepackage{color}

\def\Z{{\mathbb Z}}
\def\Q{{\mathbb Q}}

\def\P{{\mathbb P}}

\def\A{{\mathcal A}}

\def\cC{{\mathcal C}}

\def\cE{{\mathcal E}}

\def\cG{{\mathcal G}}

\def\M{{\mathcal M}}
\def\cN{{\mathcal N}}

\def\cP{{\mathcal P}}

\def\U{{\mathcal U}}

\def\G{\Gamma}

\def\d{{\mathfrak d}}
\def\g{{\mathfrak g}}

\def\n{{\mathfrak n}}
\def\p{{\mathfrak p}}
\def\r{{\mathfrak r}}

\def\u{{\mathfrak u}}

\def\etabar{{\overline{\eta}}}

\def\thetadual{\check{\theta}}

\def\Ql{{\Q_\ell}}
\def\Zl{{\Z_\ell}}

\def\Gm{{\mathbb{G}_m}}
\def\Sp{{\mathrm{Sp}}}

\def\GSp{{\mathrm{GSp}}}

\def\un{\mathrm{un}}

\def\arith{\mathrm{arith}}

\def\geom{\mathrm{geom}}

\def\ur{\mathrm{ur}}

\def\et{\mathrm{\acute{e}t}}

\def\nab{\mathrm{nab}}

\newcommand\id{\operatorname{id}}

\newcommand\Hom{\operatorname{Hom}}

\newcommand\Spec{\operatorname{Spec}}

\newcommand\Aut{\operatorname{Aut}}

\newcommand\Out{\operatorname{Out}}
\newcommand\Der{\operatorname{Der}}

\newcommand\Gr{\operatorname{Gr}}

\newcommand\Char{\operatorname{char}}

\newtheorem{theorem}{Theorem}[section]
\newtheorem{lemma}[theorem]{Lemma}
\newtheorem{proposition}[theorem]{Proposition}

\newtheorem{bigtheorem}{Theorem}
\newtheorem{bigcorollary}[bigtheorem]{Corollary}

\theoremstyle{definition}

\theoremstyle{remark}

\newtheorem{variant}[theorem]{Variant}

\begin{document}
 	
\title{Non-abelian cohomology of universal curves in positive characteristic }

\author{Tatsunari Watanabe}
\address{Mathematics Department, Embry-Riddle Aeronautical University, 3700 Willow Creek Rd. Prescott, AZ 86301, USA}
\email{watanabt@erau.edu}

\maketitle
\begin{abstract}
\textcolor{black}{
\textcolor{black}{In this paper, we will compute the non-abelian cohomology of the universal complete curve in positive characteristic. 
This extends Hain’s result on the non-abelian cohomology of generic curves in characteristic zero to positive characteristics. 
Furthermore, we will prove that the exact sequence of etale fundamental groups of the universal $n$-punctured curve in positive characteristic does not split.}
}

\smallskip

\end{abstract}


\section{Introduction}
For a DM stack $X$ with a geometric point $\bar x$, denote the \'etale fundamental group of $X$ with base point $\bar x$ by $\pi_1(X, \bar x)$. Let $F$ be a field. Fix a separable closure $\overline{F}$ of $F$. Let $C$ be a geometrically connected smooth projective curve of genus $g$ over $F$ and $\bar x$ a geometric point of $C_{\overline{F}}:= C\otimes \overline{F}$. Associated to $C$, there is the homotopy exact sequence of fundamental groups
\begin{equation}\label{hom seq for galois grp}
 1 \to \pi_1(C_{\overline{F}}, \bar x)\to \pi_1(C, \bar x)\to G_F\to 1,
 \end{equation}
where $G_F$ is the Galois group of $\overline{F}$ over $F$. Each $F$-rational point $y$ of $C$ induces a section $s_y$ of the projection $\pi_1(C, \bar x)\to G_F$ that is unique up to conjugation by an element of $\pi_1(C_{\overline{F}}, \bar x)$. Grothendieck's section conjecture predicts that if $F$ is finitely generated over $\Q$ and $g\geq 2$, then there is a bijection between the set of $F$-rational points of $C$ and the set of $\pi_1(C_{\overline{F}}, \bar x)$-conjugacy classes of \textcolor{black}{continuous} sections of $\pi_1(C, \bar x)\to G_F$. 
Assume that $2g-2+n >0$. Denote the moduli stack of curves of type $(g, n)$  with an abelian level $m$ over a field $k$ by $\M_{g,n/k}[m]$.  Our main results are concerned with the universal curves over $\M_{g,n/k}[m]$, denoted by  $\pi:\cC_{g,n/k}[m]\to \M_{g,n/k}[m]$ and $\pi^o:\M_{g,n+1/k}[m]\to \M_{g,n/k}[m]$.  \textcolor{black}{The  universal \textcolor{black}{complete} curve $\pi$ is equipped with $n$ disjoint sections called the tautological sections. }\textcolor{black}{The  universal punctured curve $\pi^o$ is the restriction of $\pi$ to the complement of the $n$ tautological sections in $\cC_{g,n/k}[m]$.}
 One of the key ingredients used in \cite{hain2} and this paper is the weighted completion of a profinite group. For $g \geq 3$ or \textcolor{black}{$g=2$ and $n>2g+2$}, let $K$ be the function field $k(\M_{g,n/k}[m])$ of $\M_{g,n/k}[m]$.  Fix a separable closure $\overline{K}$ of $K$. Let $\etabar:\Spec \overline{K}\to \M_{g,n/k}[m]$ be a geometric generic point of $\M_{g,n/k}[m]$. \textcolor{black}{Denote the fibers of $\pi$ and $\pi^o$ over $\etabar$ by $C_\etabar$ and $C^o_\etabar$, respectively.} Let $\ell$ be a prime number distinct from $\mathrm{char}(k)$. Set  $H =H^1_\et(C_{\etabar}, \Ql(1))$.  
 There is a monodromy representation $\rho_{\etabar}:\pi_1(\M_{g,n/k}, \etabar) \to \GSp(H)$.   \textcolor{black}{Fix a geometric point $\bar x$ in $C^o_\etabar$ and so in $C_\etabar$.} Denote the continuous $\ell$-adic unipotent completions of $\pi_1(C_\etabar, \bar x)$  and \textcolor{black}{$\pi_1(C^o_\etabar, \bar x)$} by $\cP$ and \textcolor{black}{$\cP^o$}
 , respectively. 
 Then there are the exact sequences of proalgebraic $\Ql$-groups
\begin{equation}\label{main exact seq for p}
1\to \cP \to \cG_{\cC_{g,n}}[\ell^r]\to \cG_{g,n}[\ell^r]\to 1,
\end{equation}
and 
\begin{equation}\label{main exact seq for p open}
\textcolor{black}{1\to \cP^o \to \cG_{g, n+1}[\ell^r]\to \cG_{g,n}[\ell^r]\to 1. }
\end{equation}
where $r$ is a nonnegative integer and \textcolor{black}{$\cG_{\cC_{g,n}}[\ell^r]$, $\cG_{g,n+1}[\ell^r]$, and $\cG_{g,n}[\ell^r]$ are the weighted completions of $\pi_1(\cC_{g,n/k}[\ell^r], \bar x)$, $\pi_1(\M_{g,n+1/k}[\ell^r], \bar x)$, and $\pi_1(\M_{g,n/k}[\ell^r], \etabar)$ with respect to $\rho_\etabar\circ \pi_\ast$,  $\rho_\etabar\circ\pi^o_\ast$, and $\rho_\etabar$, respectively. }
\textcolor{black}{Pulling back the exact sequences   (\ref{main exact seq for p})  and (\ref{main exact seq for p open}) 
along $\tilde{\rho}_{\etabar}:\pi_1(\M_{g,n/k}[\ell^r], \etabar)\to \cG_{g,n}[\ell^r](\Ql)$ induced by weighted completion, we obtain extensions
\begin{equation*} \label{main exact seq}1 \to \cP(\Ql)\to \cE_{g,n} \to \pi_1(\M_{g,n/k}[\ell^r], \etabar)\to 1\end{equation*}
and 
\begin{equation*} \label{main exact seq open}1 \to \cP^o(\Ql)\to \cE^o_{g,n+1} \to \pi_1(\M_{g,n/k}[\ell^r], \etabar)\to 1\end{equation*}
of $\pi_1(\M_{g,n/k}[\ell^r], \etabar)$ by $\cP(\Ql)$ and $\cP^o(\Ql)$, respectively.}
Denote the set of $\cP(\Ql)$-conjugacy classes of continuous sections of $\cE_{g,n}\to  \pi_1(\M_{g,n/k}[\ell^r], \etabar)$  by $H^1_\nab( \pi_1(\M_{g,n/k}[\ell^r], \etabar), \cP(\Ql))$. \textcolor{black}{Similarly, we define $H^1_\nab( \pi_1(\M_{g,n/k}[\ell^r], \etabar), \cP^o(\Ql))$ as the $\cP^o(\Ql)$-conjugacy classes of continuous sections of $\cE^o_{g,n+1}\to \pi_1(\M_{g,n/k}[\ell^r], \etabar)$. }
The non-abelian cohomology of extensions of a profinite group by a prounipotent group was introduced and developed by Kim in \cite{kim}. 
Each of the $n$ tautological sections of  $\pi$ induces a class  in $H^1_\nab( \pi_1(\M_{g,n/k}[\ell^r], \etabar), \cP(\Ql))$, which we denote by $s^\un_j$ for $j=1,\ldots,n$. 
\begin{bigtheorem}\label{nonopen case}
\textcolor{black}{Suppose that $p$ is a prime number, that $\ell$ is a prime number distinct from $p$, and that $r$ is a nonnegative integer. 
Let $k$ be a finite field with $\Char(k) =p$ that contains all $\ell^r$th roots of unity. }
If $g\geq 4$ and $n\geq 1$, then we have
$$H^1_\nab( \pi_1(\M_{g,n/k}[\ell^r], \etabar), \cP(\Ql)) = 
\{s_1^\un, \ldots,s_n^\un\}.$$
\end{bigtheorem}
\textcolor{black}{The case where $n=0$ directly follows from \cite[Prop.~11.3(ii)]{wat1}. In this case, the non-abelian cohomology of $\pi_1(\M_{g/k}[\ell^r], \etabar)$ is empty. For the universal punctured curve $\pi^o:\M_{g,n+1/k}[\ell^r]\to \M_{g,n/k}[\ell^r]$, we have the following result. }
\begin{bigtheorem}\label{open case}
\textcolor{black}{Suppose that $p$ is a prime number, that $\ell$ is a prime number distinct from $p$, and that $r$ is a nonnegative integer. 
Let $k$ be a finite field with $\Char(k) =p$ that contains all $\ell^r$th roots of unity. }
If $g\geq 4$ and $n\geq \textcolor{black}{1}$, then the sequence 
$$
1\to \cP^o \to \cG_{g, n+1}[\ell^r]\to \cG_{g,n}[\ell^r]\to 1
$$
does not split. Consequently, we have
$$
H^1_\nab( \pi_1(\M_{g,n/k}[\ell^r], \etabar), \cP^o(\Ql)) = \emptyset.
$$
\end{bigtheorem}
\textcolor{black}{There is the homotopy exact sequence associated to the universal punctured curve $\pi^o$
\begin{equation}\label{homotopy seq for univ punc curve}
1\to \pi_1(C^o_\etabar, \bar x)^{(\ell)}\to\pi_1'(\M_{g,n+1/ k}[\ell^r],\bar x )\to\pi_1(\M_{g,n/ k}[\ell^r], \etabar)\to1,
\end{equation}
where $\pi_1(C^o_\etabar, \bar x)^{(\ell)}$ is the pro-$\ell$ completion of $\pi_1(C^o_\etabar, \bar x)$ and the middle group is the quotient of $\pi_1(\M_{g,n+1/ k}[\ell^r],\bar x )$ by a certain distinguished subgroup (see \cite[SGA 1, Expos{\'e} XIII, \S4]{sga1} and Prop.~\ref{exact seq for punctured universal curve}). Since the sequence (\ref{main exact seq for p open}) is induced from the sequence (\ref{homotopy seq for univ punc curve})  by  weighted completion and a splitting of  (\ref{homotopy seq for univ punc curve}) induces that of (\ref{main exact seq for p open}) by a universal property of weighted completion, we have an immediate consequence. 
\begin{bigcorollary}If $g \geq 4$ and $n \geq 1$, then the sequence (\ref{homotopy seq for univ punc curve}) does not split. 
\end{bigcorollary}}
\textcolor{black}{
Theorem \ref{nonopen case} can be viewed as an analogue of Hain's result on generic curves in characteristic zero \cite[Thm.~3]{hain2} for the universal complete curve in positive characteristic. On the other hand, Theorem 2 is the positive characteristic analogue of the author's result in \cite[Thm.~1]{wat2}.   In \cite{wat1}, the author proves that the rational points of the universal complete  curve in positive characteristic are given by exactly the tautological sections and that the exact sequence (\ref{hom seq for galois grp}) associated to the function field $K$ of $\M_{g/k}[\ell^r]$ does not split when $n =0$. From the \textcolor{black}{point of }view of anabelian geometry, our main results provide more evidence for the bijection between the set of the rational points of the universal complete  curve and the set of $\pi_1(C_\etabar, \bar x)$-conjugacy classes of continuous sections of $\pi_1(\cC_{g,n/k}[\ell^r], \bar x)\to \pi_1(\M_{g,n/k}[\ell^r], \etabar)$.
The main new ingredients used in this paper are the non-abelian cohomology of extensions of a profinite group by a prounipotent group introduced by Kim in \cite{kim} and  the non-abelian cohomology schemes of weighted completions developed by Hain in \cite{hain4}. In our case where $k$ is a finite field, Proposition \ref{condition for existence} allows us to use a key exact sequence for the non-abelian cohomology schemes to compute the non-abelian cohomology of  $\pi_1(\M_{g,n/k}[\ell^r], \etabar)$. }
\section{The universal curve of type $(g, n)$ and level structures}
Let  $T$ be a scheme. By a curve of type $(g,n)$ over $T$, we mean a smooth proper morphism $f:C\to T$  whose geometric fibers are connected one-dimensional schemes of arithmetic genus g, that is equipped with $n$ disjoint sections. For nonnegative integers $g$ \textcolor{black}{and} $n$ satisfying $2g-2+n>0$, we have the smooth Deligne-Mumford (DM) stack $\M_{g,n}$ over $\Spec \Z$ classifying curves of type $(g, n)$. For a field $k$, the stack $\M_{g,n/k}$ is the base change $\M_{g,n}\otimes k$. The universal curve $\pi:\cC_{g,n}\to \M_{g,n}$ exists and satisfies the property that for a curve $f:C\to T$ of type $(g, n)$, there exists a unique morphism $\psi_f:T\to \M_{g,n}$ such that $f$ is the pullback of $\pi$ along $\psi_f$. \\
\indent  Assume that $2g-2+n >0$. Let $m$ be a positive integer. Let $k$ be a field containing all $m$th roots of unity $\mu_m(\bar k)$ \textcolor{black}{with $(\Char(k), m) =1$}. Fix an isomorphism $\mu: \mu_m^{\otimes-1}\cong \Z/m\Z$. For a curve $f:C\to T$ of type $(g, n)$, an abelian level $m$ on $f$ is an isomorphism $\phi:R^1f_\ast(\Z/m\Z) \cong (\Z/m\Z)^{2g}$ such that the diagram
$$\xymatrix@C=1pc @R=1pc{
		\Lambda^2R^1f_\ast(\Z/m\Z) \ar[r]^-{\mathrm{cup}}\ar[d]_{\Lambda^2\phi}&  R^2f_\ast(\Z/m\Z)\cong \mu_m^{\otimes-1}\ar[d]^{\mu}\\
		\Lambda^2(\Z/m\Z)^{2g} \ar[r]& \Z/m\Z
	}$$
commutes, where $R^1f_\ast(\Z/m\Z)$  is equipped with a symplectic structure via the cup product and $(\Z/m\Z)^{2g}$ is equipped with the standard symplectic structure.  The moduli stack of curves of type $(g,n)$ with an abelian level $m$ over $k$ is denoted by $\M_{g, n/k}[m]$.  \textcolor{black}{It is a geometrically connected, finite \'etale  cover of $\M_{g,n/k}$ (see \cite{DM}).
 } When $m \geq 3$, it is a smooth scheme over $k$.  In this paper, we always assume that $k$ contains all $m$th roots of unity. \textcolor{black}{For $m =1$, we denote $\M_{g,n/k}[1]$ by $\M_{g,n/k}$. }
\section{Weighted Completion and its Application to the Universal Curves}\label{weight}
\subsection{The weighted completion of a profinite group}
The weighted completion of a profinite group is introduced and developed by Hain and Matsumoto.  A detailed introduction of the theory and properties are included in \cite{wei}. Here, we briefly review the definition and list some key properties needed in this paper. 
Let $F$ be a field of characteristic zero. Suppose that $\G$ is a profinite group, $R$ is a reductive group over $F$, $\omega: \Gm\to R$ is a nontrivial central cocharacter, and that $\rho:\G\to R(F)$ is a continuous Zariski-dense homomorphism.
%
A negatively weighted extension of $R$ is an (pro)algebraic group $G$ that is an extension of $R$ by a (pro)unipotent group $U$ over $F$, $1\to U\to G\to R\to 1$,
such that $H_1(U)$ admits only negative weights as a $\Gm$-representation via the central cocharacter $\omega$.
The {\it weighted completion} of $\G$ with respect to $\rho$ and $\omega$ consists of a proalgebraic $F$-group $\cG$ that is a negatively weighted extension of $R$ and a homomorphism $\tilde{\rho}:\G\to\cG(F)$ lifting $\rho$ that admits a universal property: if $G$ is a negatively weighted extension of $R$ and there is a Zariski-dense homomorphism $\rho_G:\G\to G(F)$ lifting $\rho$, then there is a unique morphism $\phi_G:\cG\to G$ such that 
$$\rho_G=\phi_G\circ \tilde{\rho}.$$
Denote the prounipotent radical of $\cG$ by $\U$. \textcolor{black}{Since $\U$ is prounipotent, by a generalization of Levi's Theorem, the extension $1\to \U\to \cG\to R\to 1$ splits, and any two splittings are conjugate by an element of $\U(F)$.}
Furthermore, there is a natural weight filtration $W_\bullet M$ on a finite dimensional $\cG$-module $M$. The following is the list of the key properties.
\begin{proposition}[{\cite[\textcolor{black}{Prop.~3.8, Thms~3.9 \& 3.12 }]{wei}}]\label{weight filt} With the notation as above, the weight filtration $W_\bullet M$ satisfies the properties:
\begin{enumerate}
\item The defining $\cG$-action preserves the weight filtration.
\item The $\cG$-action on each associated graded quotient $\Gr^W_nM$ factors through $R$.
\item The functors $M\mapsto \Gr_\bullet M$, $M\mapsto W_mM$, and $M\mapsto M/W_mM$ are exact on the category of finite dimensional $\cG$-modules.
\end{enumerate}

\end{proposition}
We immediately see that the results of the proposition extend to direct and inverse limits of finite dimensional $\cG$-modules. 
\subsection{Applications to the universal curves}\label{Applications to universal curves}
  Let $k$ be a finite field of characteristic $p$. Let $g\geq 3$, $n\geq 0$, and $m \geq 1$ be prime to $p$. 
For  a prime $\ell$ distinct from $p$ and \textcolor{black}{$A=\Zl$ or $\Ql$}, set $H_A=H^1_\et(C_\etabar, A(1))$, where $C_\etabar$ is the fiber of the universal curve $\pi: \cC_{g,n/k}[m]\to \M_{g,n/k}[m]$ over $\etabar$. For simplicity, when $A =\Ql$, we denote $H_\Ql$ by $H$. 
Since the $\ell$-adic sheaf $R^1\pi_\ast \Zl(1)$ is smooth over $\M_{g,n/k}[m]$, we obtain a representation $\rho_\etabar: \pi_1(\M_{g,n/k}[m], \etabar)\to \GSp(H_{\Zl})$.  Note that when $k$ is a finite field \textcolor{black}{with $\Char(k)\not =\ell$}, a number field, or \textcolor{black}{a local field $\Q_q$}, the $\ell$-adic cyclotomic character $\chi_\ell:G_k \to \Zl^\times$ has infinite image. 
\begin{proposition}[{\cite[\textcolor{black}{Prop.~8.3}]{wat1}}]\label{monodromy density} 
  Let $k$ be a finite field of characteristic $p$.  If $g\geq 3$, $n\geq 0$,  $\ell$ is a prime number distinct from $p$, and $r\geq 0$,  then the image of the monodromy representation 
 $$\rho_\etabar:\pi_1(\M_{g,n/k}[\ell^r], \etabar)\to \GSp(H)$$ is Zariski-dense in $\GSp(H)$. 

\end{proposition}
%
%
Define a central cocharacter $\omega:\Gm\to \GSp(H)$ by sending $z$ to $z^{-1}\id$. 
Let $r$ be a nonnegative integer. Denote the weighted completion of $\pi_1(\M_{g,n/k}[\ell^r],\etabar)$ with respect to $\rho_{\etabar}$ and $\omega$ by 
$$(\cG_{g,n}[\ell^r],\,\, \tilde{\rho}_{\etabar}:\pi_1(\M_{g,n/k}[\ell^r],\etabar)\to \cG_{g,n}[\ell^r](\Ql)).$$
The completion $\cG_{g,n}[\ell^r]$ is a negatively weighted extension of $\GSp(H)$ by a prounipotent $\Ql$-group \textcolor{black}{$\U_{g,n}[\ell^r]$. The} Lie algebras of $\cG_{g,n}[\ell^r]$ and $\U_{g,n}[\ell^r]$ are denoted by $\g_{g,n}[\ell^r]$ and $\u_{g,n}[\ell^r]$, respectively. 
\begin{variant} Let $\bar x$ be a geometric point in the fiber $C^o_\etabar$. We also consider $\bar x$ as a geometric point in $\cC_{g,n/k}[\ell^r]$. Let $(\cG_{\cC_{g,n}}[\ell^r], \tilde{\rho}^\cC_\etabar:\pi_1(\cC_{g,n/k}[\ell^r], \bar x) \to \cG_{\cC_{g,n}}[\ell^r](\Ql))$ be the weighted completion of $\pi_1(\cC_{g,n/k}[\ell^r], \bar x)$ with respect to the representation $\rho^\cC_{\etabar}:=\rho_\etabar\circ\pi_\ast :\pi_1(\cC_{g,n/k}[\ell^r], \bar x) \to \pi_1(\M_{g,n/k}[\ell^r], \etabar)\to \GSp(H)$ and $\omega$. Denote the prounipotent radical of $\cG_{\cC_{g,n}}[\ell^r]$ by $\U_{\cC_{g,n}}[\ell^r]$.   The Lie algebras of $\cG_{\cC_{g,n}}[\ell^r]$ and $\U_{\cC_{g,n}}[\ell^r]$ are denoted by $\g_{\cC_{g,n}}[\ell^r]$ and $\u_{\cC_{g,n}}[\ell^r]$, respectively.
\end{variant}
\subsection{Relative completion of $\pi_1(\M_{g,n/\bar k}[\ell^r])$ and the $\cG_{g,n}[\ell^r]$-module $\u^\geom_{g,n}[\ell^r]$} 
Let $\bar k$  be the separable closure of $k$ in $\overline{K}$, where $k$ is as above.   
Another key ingredient in this paper is the Lie algebra of the relative completion of $\pi_1(\M_{g,n/\bar k}[\ell^r], \etabar)$, which admits a $\cG_{g,n}[\ell^r]$-module structure and hence a weight filtration. A detailed review and basic properties of relative completion are given in \cite{hain5}.
Recall that when $\ell$ is different from $\mathrm{char}(k)$, the image of the  monodromy representation $\rho^\geom_{\etabar}: \pi_1(\M_{g,n/\bar k}[\ell^r], \etabar)\to \Sp(H)$ is Zariski-dense in $\Sp(H)$.
The relative completion of $\pi_1(\M_{g,n/\bar k}[\ell^r], \etabar)$ with respect to $\rho^\geom_{\etabar}$ consists of a proalgebraic $\Ql$-group $\cG^{\geom}_{g,n}[\ell^r]$ and a Zariski-dense homomorphism $\tilde{\rho}^\geom_{\etabar}:\pi_1(\M_{g,n/\bar k}[\ell^r], \etabar)\to \cG^{\geom}_{g,n}[\ell^r](\Ql)$. The proalgebraic group $\cG^{\geom}_{g,n}[\ell^r]$ is an extension of $\Sp(H)$ by a prounipotent $\Ql$-group \textcolor{black}{$\U^{\geom}_{g,n}[\ell^r]$, 
and} the homomorphism $\tilde{\rho}^\geom_{\etabar}$ satisfies a universal property:
if $G$ is an extension of $\Sp(H)$ by a prounipotent group $U$ and there is a Zariski-dense homomorphism $\rho^\geom_G: \pi_1(\M_{g,n/\bar k}[\ell^r], \etabar)\to G(\Ql)$ lifting $\rho^\geom_{\etabar}$, then there is a unique morphism $\phi_G:\cG^{\geom}_{g,n}[\ell^r]\to G$ such that 
$\rho^\geom_G=\phi_G\circ \tilde{\rho}^\geom_{\etabar}.$
Denote the Lie algebras of $\cG^{\geom}_{g,n}[\ell^r]$, $\U^{\geom}_{g,n}[\ell^r]$, and $\Sp(H)$ by $\g^{\geom}_{g,n}[\ell^r], \u^{\geom}_{g,n}[\ell^r]$, and $\r$, respectively. \textcolor{black}{For $r =0$, we denote $\cG^\geom_{g,n}[1]$, $\U^\geom_{g,n}[1]$, $\g^\geom_{g,n}[1]$, and $\u^\geom_{g,n}[1]$ by $\cG^\geom_{g,n}$, $\U^\geom_{g,n}$, $\g^\geom_{g,n}$, and $\u^\geom_{g,n}$, respectively. }
We list results needed in this paper:

\begin{proposition}[{ \cite[Thm.~3.9 ]{wei}, \cite[Thm.~5.6, Props.~8.4, ~8.6 \& ~8.7]{wat1}}] \label{comparison}With notation as above, \textcolor{black}{when $g\geq 3$,} we have the following results:
\begin{enumerate}
\item The conjugation action of $\pi_1(\M_{g,n/k}[\ell^r], \etabar)$ on $\pi_1(\M_{g,n/\bar k}[\ell^r], \etabar))$ induces the adjoint action of $\cG_{g,n}[\ell^r]$ on $\g^{\geom}_{g,n}[\ell^r]$, and hence on $\u^{\geom}_{g,n}[\ell^r]$. Therefore, the Lie algebras $\g^{\geom}_{g,n}[\ell^r]$ and $\u^{\geom}_{g,n}[\ell^r]$ admit weight filtrations $W_\bullet \g^{\geom}_{g,n}[\ell^r]$ and $W_\bullet \u^{\geom}_{g,n}[\ell^r] $, respectively.
\item There is an isomorphism $\cG^\geom_{g,n}[\ell^r]\cong \cG^\geom_{g,n}$ and 
there are weight filtration preserving isomorphisms 
$$\g^\geom_{g,n}[\ell^r]\cong \g^\geom_{g,n} \,\,\text{ and }\,\, \u^\geom_{g,n}[\ell^r] \cong \u^\geom_{g,n}$$
that induce $\GSp(H)$-equivariant isomorphisms of graded Lie algebras
$$\Gr^W_\bullet\g^{\geom}_{g,n}[\ell^r]\cong\Gr^W_\bullet\g^{\geom}_{g,n} \,\,\text{ and }\,\, \Gr^W_\bullet\u^\geom_{g,n}[\ell^r] \cong \Gr^W_\bullet\u^\geom_{g,n}.$$
\item There is a weight filtration preserving  isomorphism between $\g^\geom_{g,n}$ and \textcolor{black}{the Lie algebra of }the relative completion of $\pi_1(\M_{g,n/\bar\Q}, \etabar')$ with respect to $\rho^\geom_{\etabar'}:\pi_1(\M_{g,n/\bar\Q}, \etabar')\to \Sp(H^1_\et(C_{\etabar'}, \Ql(1)))$, where $\etabar'$ is a geometric generic point of $\M_{g,n/\bar\Q}$. 
\end{enumerate}
\end{proposition}
\subsubsection{Variant} \textcolor{black}{Denote by $\cG^{\geom}_{\cC_{g,n}}[\ell^r]$ the relative completion  of $\pi_1(\cC_{g,n/\bar k}[\ell^r], \bar x)$ with respect to the composition $\pi_1(\cC_{g,n/\bar k}[\ell^r], \bar x)\to \pi_1(\M_{g,n/\bar k}[\ell^r], \etabar)\to \Sp(H)$. It is an extension of $\Sp(H)$ by a prounipotent $\Ql$-group denoted by $\U^{\geom}_{\cC_{g,n}}[\ell^r]$. Denote the Lie algebras of $\cG^{\geom}_{\cC_{g,n}}[\ell^r]$ and $\U^{\geom}_{\cC_{g,n}}[\ell^r]$ by $\g^{\geom}_{\cC_{g,n}}[\ell^r]$ and $\u^{\geom}_{\cC_{g,n}}[\ell^r]$, respectively.  
As above, for $r =0$, we denote $\cG^\geom_{\cC_{g,n}}[1]$, $\U^\geom_{\cC_{g,n}}[1]$, $\g^\geom_{\cC_{g,n}}[1]$, and $\u^\geom_{\cC_{g,n}}[1]$ by $\cG^\geom_{\cC_{g,n}}$, $\U^\geom_{\cC_{g,n}}$, $\g^\geom_{\cC_{g,n}}$, and $\u^\geom_{\cC_{g,n}}$, respectively.  The analogue of Proposition \ref{comparison} also holds for $\g^{\geom}_{\cC_{g,n}}[\ell^r]$ and $\u^{\geom}_{\cC_{g,n}}[\ell^r]$ (see \cite[\S 8.1]{wat1}). }
\subsection{Exact sequences} \label{exact seqs}The relative and weighted completions associated to the complete universal curve $\pi$ fit in the following commutative diagram.  \textcolor{black}{Let $\cP$ be the continuous unipotent completion of $\pi_1(C_\etabar, \bar x)$ over $\Ql$. It is a prounipotent group over $\Ql$.} Then for $g\geq 2$, by \cite[Prop.~7.6]{wat1} there is a commutative diagram of completions ($a$)\label{key seq}
 $$\xymatrix@C=1pc @R=1pc{
		1\ar[r]& \cP\ar[r]\ar@{=}[d] &\cG^\geom_{\cC_{g,n}}[\ell^r]\ar[r]\ar[d]&\cG^\geom_{g,n}[\ell^r]\ar[r]\ar[d]&1\\
		1\ar[r] &\cP\ar[r] &\cG_{\cC_{g,n}}[\ell^r]\ar[r]&\cG_{g,n}[\ell^r]\ar[r]&1,
	}
	$$
 where the rows are exact and the middle and right-hand vertical maps are induced by the natural map $\M_{g,n/\bar k}[\ell^r]\to \M_{g,n/k}[\ell^r]$ given by base change. \textcolor{black}{Let $\cP^o$ be the continuous unipotent completion of $\pi_1(C^o_\etabar, \bar x)$ over $\Ql$. It is a prounipotent group over $\Ql$.}  Similarly, we have the following result for the  universal punctured curve $\pi^o$.
 
 \begin{proposition}\label{exact seq for punctured universal curve}
 With notation as above, if $g\geq 2$ and $n\geq1$, then there is a commutative diagram of completions associated to $\pi^o:\M_{g,n+1/k}\to \M_{g,n/k}$:
  $$\xymatrix@C=1pc @R=1pc{
		1\ar[r]& \cP^o\ar[r]\ar@{=}[d] &\cG^\geom_{g,n+1}[\ell^r]\ar[r]\ar[d]&\cG^\geom_{g,n}[\ell^r]\ar[r]\ar[d]&1\\
		1\ar[r] &\cP^o\ar[r] &\cG_{g,n+1}[\ell^r]\ar[r]&\cG_{g,n}[\ell^r]\ar[r]&1.
	}
	$$

 \end{proposition}
 
 \begin{proof}
 Consider the commutative diagram of fundamental groups
  $$\xymatrix@C=1pc @R=1pc{
		 &\pi_1(C^o_\etabar, \bar x)\ar[r]\ar@{=}[d] &\pi_1(\M_{g,n+1/\bar k}[\ell^r],\bar x )\ar[r]\ar[d]&\pi_1(\M_{g,n/\bar k}[\ell^r], \etabar)\ar[r]\ar[d]&1\\
		&\pi_1(C^o_\etabar, \bar x)\ar[r] \ar@{=}[d] &\pi_1(\M_{g,n+1/ k}[\ell^r], \bar x)\ar[r]\ar[d]&\pi_1(\M_{g,n/k}[\ell^r], \etabar)\ar[r]\ar[d]&1\\
 &\pi_1(C^o_\etabar, \bar x)\ar[r] &\pi_1(\M_{g,n+1/\Z[1/\ell]},\bar x )\ar[r]&\pi_1(\M_{g,n/\Z[1/\ell]}, \etabar)\ar[r]&1.		
		}
	$$
Let $Y$ be the kernel of 	$\pi_1(\M_{g,n+1/\bar k}[\ell^r],\bar x )\to \pi_1(\M_{g,n/\bar k}[\ell^r], \etabar)$, and let $W$ be the kernel of the projection $Y\to Y^{(\ell)}$, where $Y^{(\ell)}$ is the pro-$\ell$ completion of $Y$. Define $\pi_1'(\M_{g,n+1/\bar k}[\ell^r],\bar x )$ as the quotient of $\pi_1(\M_{g,n+1/\bar k}[\ell^r],\bar x )$ by $W$. Define $\pi_1'(\M_{g,n+1/k}[\ell^r],\bar x )$  and $\pi_1'(\M_{g,n+1/\Z[1/\ell]},\bar x )$ in a similar manner.
Then the rows of the commutative diagram ($\ast$)
	 $$\xymatrix@C=1pc @R=1pc{
		&\pi_1(C^o_\etabar, \bar x)^{(\ell)}\ar[r]\ar@{=}[d] &\pi_1'(\M_{g,n+1/\bar k}[\ell^r],\bar x )\ar[r]\ar[d]&\pi_1(\M_{g,n/\bar k}[\ell^r], \etabar)\ar[r]\ar[d]&1\\
		&\pi_1(C^o_\etabar, \bar x)^{(\ell)}\ar[r]\ar@{=}[d] &\pi_1'(\M_{g,n+1/ k}[\ell^r], \bar x)\ar[r]\ar[d]&\pi_1(\M_{g,n/k}[\ell^r], \etabar)\ar[r]\ar[d]&1\\
		 &\pi_1(C^o_\etabar, \bar x)^{(\ell)}\ar[r] &\pi_1'(\M_{g,n+1/\Z[1/\ell]},\bar x )\ar[r]&\pi_1(\M_{g,n/\Z[1/\ell]}, \etabar)\ar[r]&1
	}
	$$
	are exact (see \cite[SGA 1, Expos{\'e} XIII, 4.3 and 4.4]{sga1}). Moreover, there is a commutative diagram ($\ast\ast$) with exact rows
 $$\xymatrix@C=1pc @R=1pc{
	&\pi_1(C^o_\etabar, \bar x)^{(\ell)}\ar[r]\ar@{=}[d] &\pi_1'(\M_{g,n+1/ \Z[1/\ell]}, \bar x)\ar[r]\ar[d]&\pi_1(\M_{g,n/\Z[1/\ell]}, \etabar)\ar[r]\ar[d]&1\\
		1\ar[r] &\pi_1(C^o_\etabar, \bar x)^{(\ell)}\ar[r] &\pi_1(\cC(\M_{g, n+1/\Z[1/\ell]}),F_{\bar x} )\ar[r]&\pi_1(\cC(\M_{g, n/\Z[1/\ell]}), F_\etabar)\ar[r]&1,
	}
	$$	
where $\cC(\M_{g, n/\Z[1/\ell]})$ is the Galois category of geometrically relative-$\ell$ coverings over $\M_{g, n/\Z[1/\ell]}$ defined in \cite[\S 7]{relative prol} and where $F_{\bar x}$ and $F_\etabar$ are fiber functors over $\bar x$ and $\etabar$, respectively. 
	The left exactness of the bottom row of the diagram $(\ast\ast)$ follows from 
	\cite[Prop.~3.1 (2)]{relative prol} and the proof of \cite[Thm.~7.6]{relative prol}. Hence the top row of the diagram ($\ast\ast$) is left exact. Therefore, the first and middle rows of the diagram $(\ast)$ are left exact as well.  By the center freeness of $\cP^o$ (e.g. \cite{ntu}) and the exactness criterion for weighted completions \cite[Prop.~6.11]{hain2}, the bottom row of the diagram
	 $$\xymatrix@C=1pc @R=1pc{
		& \cP^o\ar[r]\ar@{=}[d] &\cG^\geom_{g,n+1}[\ell^r]\ar[r]\ar[d]&\cG^\geom_{g,n}[\ell^r]\ar[r]\ar[d]&1\\
		1\ar[r] &\cP^o\ar[r] &\cG_{g,n+1}[\ell^r]\ar[r]&\cG_{g,n}[\ell^r]\ar[r]&1.
	}
	$$
is exact. The right exactness of relative completions \cite[Prop.~3.7]{hain5} implies that the top row is exact. The exactness of the bottom row implies that the top row is left exact as well. 
 \end{proof}

\subsection{$\Gr^W_\bullet\u_{g,n}[\ell^r]/W_{-3}$} The Lie algebra $\u_{g,n}[\ell^r]$ is a $\cG_{g,n}[\ell^r]$-module via the adjoint action and hence admits a weight filtration $W_\bullet\u_{g,n}[\ell^r]$. We study the relation between $\Gr^W_\bullet\u^\geom_{g,n}[\ell^r]/W_{-3}$ and $\Gr^W_\bullet\u_{g,n}[\ell^r]/W_{-3}$. The case when $\mathrm{char}(k)=0$ is determined in \cite[\S 8]{hain2}. First, for a finite field $k$, we study the weighted completion of $G_k$. Let $\chi_\ell:G_k\to \Z^\times_\ell \subset \Gm(\Ql)$ be the $\ell$-adic cyclotomic character.   Define a central cocharacter $\omega_\ell:\Gm\to \Gm$ by mapping $z\mapsto z^{-2}$. Denote the weighted completion of $G_k$ with respect to $\chi_\ell$ and $\omega_\ell$ by ($\A_k$, $\tilde{\chi}_\ell:G_k\to \A_k(\Ql)$) and denote the unipotent radical of $\A_k$ by $\cN_k$. Denote the Lie algebra of $\cN_k$  by $\n_k$.
Let $\tau:\GSp(H)\to \Gm$ be the similitude character.   There is a commutative diagram 
$$\xymatrix@C=1pc @R=1pc{
		1\ar[r] &\pi_1(\M_{g,n/\bar k}[\ell^r], \etabar)\ar[r]\ar[d]^{\rho^\geom_\etabar} &\pi_1(\M_{g,n/k}[\ell^r], \etabar)\ar[r]\ar[d]^{\rho_\etabar}&G_k\ar[r]\ar[d]^{\chi_\ell}&1\\
		1\ar[r] &\Sp(H)\ar[r] &\GSp(H)\ar[r]^\tau&\Gm\ar[r]&1\\
		& &\Gm\ar[u]_\omega\ar@{=}[r] &\Gm\ar[u]_{\omega_\ell}&
	}
	$$
	whose rows are exact. Note that $\tau\circ \omega =\omega_\ell$.
In the above diagram, applying weighted completion to the middle and right columns and relative completion to the left column produces a commutative diagram 
$$\xymatrix@C=1pc @R=1pc{
		&\cG^\geom_{g,n}[\ell^r]\ar[r]\ar[d]&\cG_{g,n}[\ell^r]\ar[r]\ar[d]&\A_k\ar[r]\ar[d]&1\\
		1\ar[r] &\Sp(H)\ar[r] &\GSp(H)\ar[r]^\tau&\Gm\ar[r]&1.
	}
	$$
The right exactness of relative and weighted completion implies that the top row is exact. In our case, we have the following result.
	\begin{proposition}\label{weighted comp for G_k}
	With the notation as above, we have $\A_k =\Gm$. 
	\end{proposition}
\begin{proof} By \cite[Prop.~6.8]{hain2}, there is an isomorphism $\Hom_\Gm(H_1(\n_k), V) \cong H^1(G_k, V)$ for each finite dimensional $\Gm$-module $V$ of weight $m < 0$. Thus it suffices to show that $H^1(G_k, V)=0$. Let $I_\ell$ and $K_\ell$ be the image and kernel of $\chi_\ell$, respectively.  Since $k$ is a finite field, the image $I_\ell$ is infinite. Associated to the exact sequence,
$1\to K_\ell\to G_k\overset{\chi_\ell}\to I_\ell\to1,$
there is an exact  sequence
$$0\to H^1(I_\ell, H^0(K_\ell, V))\to H^1(G_k, V) \hspace{2in}$$
$$\hspace{1in}\to H^0(I_\ell, H^1(K_\ell, V)) \to H^2(I_\ell, H^0(K_\ell, V)).$$
For $m\not=0$, the first and fourth terms vanish (see \cite[Ex.~4.12]{wei}). Since $K_\ell$ acts trivially on $V$, there is an isomorphism  $H^0(I_\ell, H^1(K_\ell, V))\cong \Hom_{I_\ell}(K_\ell, V)$. The Galois group $G_k\cong \widehat{\Z}$ is abelian, and so the induced action of $I_\ell$  on $K_\ell$ is trivial. Let $\phi \in \Hom_{I_\ell}(K_\ell, V)$ and $x\in K_\ell$. For $a \in I_\ell$, we have 
$\phi(x) =\phi(ax) =  a^m\phi(x).$
For $m\not= 0$, we can find $a$ such that $1-a^m$ is in $\Q^\times_\ell$ since $I_\ell$ is infinite, so $\phi(x) =0$. Thus it follows that $\Hom_{I_\ell}(K_\ell, V)=0$. Therefore, $H^1(G_k, V)$ vanishes. 
\end{proof}
By Proposition \ref{weighted comp for G_k}, $\cN_k$ is trivial and so the natural map $\cG^\geom_{g,n}[\ell^r]\to\cG_{g,n}[\ell^r]$ induces   a surjection $\beta: \U^\geom_{g,n}[\ell^r]\to \U_{g,n}[\ell^r]$ and hence a surjection of pronilpotent Lie algebras $d\beta: \u^\geom_{g,n}[\ell^r]\to\u_{g,n}[\ell^r]$. Applying the functor $\Gr^W_\bullet$ gives a surjection of graded Lie algebras $\Gr(d\beta): \Gr^W_\bullet\u^\geom_{g,n}[\ell^r]\to \Gr^W_\bullet\u_{g,n}[\ell^r]$. 
\begin{proposition}\label{weight -1 and -2 injection}
For $g \geq 3$ and $n\geq 0$, the graded Lie algebra map $\Gr(d\beta)$ induces $\GSp(H)$-equivariant isomorphisms
$$\Gr^W_m\u^\geom_{g,n}[\ell^r]\cong \Gr^W_m\u_{g,n}[\ell^r]\,\,\text{ for } m=-1,-2.$$ 
\end{proposition}
\begin{proof} It suffices to show that $\Gr(d\beta)$ is injective for weight $m=-1,-2$. 
 The Lie algebra of $\cP$ denoted by $\p$ admits a weight filtration via the adjoint action of $\cG_{g,n}[\ell^r]$ on $\p$ and so does the derivation Lie algebra of $\p$ denoted by $\Der \p$.  Assume $n\geq 1$. In the diagram $(a)$ in \S\ref{key seq}, each tautological section $s_j$ of $\cG_{\cC_{g,n}}[\ell^r]\to\cG_{g,n}[\ell^r]$ gives  the weight filtration preserving adjoint action that produces the commutative diagram of graded Lie algebras  
 $$\xymatrix@C=1pc @R=1pc{
	\mathrm{adj}_j:\Gr^W_\bullet\u^\geom_{g,n}[\ell^r]\ar[r]\ar[d]^{\Gr(d\beta)}& \Gr^W_\bullet\Der \p \ar@{=}[d] \\
	 \mathrm{adj}_j:\Gr^W_\bullet\u_{g,n}[\ell^r]\ar[r]&\Gr^W_\bullet\Der\p.
		}
	$$
Recall that $k$ is a finite field of characteristic $p$. Let $\etabar'$ be a geometric generic point of $\M_{g,n/\bar\Q_p}[\ell^r]$ and $C_{\etabar'}$ the fiber of the universal curve $\cC_{g,n/\bar\Q_p}[\ell^r]\to\M_{g,n/\bar\Q_p}[\ell^r]$ over $\etabar'$. Denote the Lie algebra of the $\ell$-adic unipotent completion of $\pi_1(C_{\etabar'}, \bar x')$ by $\p'$. Denote the relative completion of $\pi_1(\M_{g,n/\bar \Q_p}[\ell^r], \etabar')$ with respect to the monodromy representation $\rho^{\bar\Q_p}_{\etabar'}: \pi_1(\M_{g,n/\bar \Q_p}[\ell^r], \etabar') \to \Sp(H^1_\et(C_{\etabar'}, \Ql(1)))$ by $\cG^{\bar\Q_p}_{g,n}[\ell^r]$. Denote the prounipotent radical of $\cG^{\bar\Q_p}_{g,n}[\ell^r]$ by $\U^{\bar\Q_p}_{g,n}[\ell^r]$ and its Lie algebra by $\u^{\bar\Q_p}_{g,n}[\ell^r]$. Fix an isomorphism $\gamma: \pi_1(\M_{g,n/\Z^\ur_p}[\ell^r], \etabar)\cong \pi_1(\M_{g,n/\Z^\ur_p}[\ell^r], \etabar')$, where $\Z^\un_p$ is the maximal unramified extension of $\Z_p$. The isomorphism $\gamma$ induces weight filtration preserving isomorphisms $\u^\geom_{g,n}[\ell^r]\cong \u^{\bar\Q_p}_{g,n}[\ell^r]$, $\p\cong \p'$, and $\Der\p\cong \Der\p'$ that make the diagram
$$\xymatrix@C=1pc @R=1pc{
		\prod_{j=1}^n\Gr^W_\bullet\mathrm{adj}_j:&\Gr^W_\bullet\u^\geom_{g,n}[\ell^r]\ar[r]\ar[d]& \bigoplus_{j=1}^n\Gr^W_\bullet\Der \p \ar[d] \\
		\prod_{j=1}^n\Gr^W_\bullet\mathrm{adj}_j:&\Gr^W_\bullet\u^{\bar\Q_p}_{g,n}[\ell^r]\ar[r]&\bigoplus_{j=1}^n\Gr^W_\bullet\Der\p',
		}
	$$
commute, where the vertical maps are isomorphisms of graded Lie algebras induced $\gamma$. For example, see \cite[\S8.1]{wat1}. From \cite[Prop.~8.6 \& 8.8]{hain2}, it follows that the bottom adjoint map is injective for weight $m=-1, -2$ and so is the adjoint map on the top row.  Therefore, for $n\geq 1$, the graded Lie algebra map $\Gr(d\beta)_m :\Gr^W_m\u^\geom_{g,n}[\ell^r]\to \Gr^W_m\u_{g,n}[\ell^r]$ is injective for $m=-1,-2$. For $n=0$, consider the commutative diagram
$$\xymatrix@C=1pc @R=1pc{
		0\ar[r]&\Gr^W_\bullet\p\ar[r]\ar@{=}[d]&\Gr^W_\bullet\u^\geom_{g,1}[\ell^r]\ar[r]\ar[d]&\Gr^W_\bullet\u^\geom_{g}[\ell^r]\ar[r]\ar[d]&0\\
		0\ar[r]&\Gr^W_\bullet\p\ar[r]&\Gr^W_\bullet\Der\p\ar[r]&\Gr^W_\bullet\Out\Der\p\ar[r]&0,
}
	$$
	where the injection $\Gr^W_\bullet\p \to \Gr^W_\bullet\Der\p$ is the inner adjoint. It follows that the right-hand vertical map is injective for weight $m=-1,-2$. Since the outer action $\Gr^W_\bullet\u^\geom_g[\ell^r]\to \Gr^W_\bullet \Out\Der\p$ factors through $\Gr^W_\bullet\u_g[\ell^r]$, it follows that  the map $\Gr^W_m\u^\geom_g[\ell^r]\to \Gr^W_m\u_g[\ell^r]$ is injective for $m=-1,-2$. 
\end{proof}
Together with the diagram $(a)$ in \S\ref{key seq}, Propositions \ref{comparison} and \ref{weight -1 and -2 injection}  give the following result.
\begin{proposition}\label{iso in weight -1 and -2}For $g\geq 3$ and $\u\in \{\u_{g,n}, \u_{\cC_{g,n}}\} $ there is a $\GSp(H)$-equivariant graded Lie algebra isomorphism
$$\Gr^W_{\bullet}\u/W_{-3}\cong \Gr^W_{\bullet}\u^\geom/W_{-3}.$$
\end{proposition}
\subsection{$\d_{g,n}$} \label{two-step} Here we review the two-step graded nilpotent Lie algebra $\d_{g,n}$ associated to the universal curve $\pi:\cC_{g,n/k}[\ell^r]\to \M_{g,n/k}[\ell^r]$. The Lie algebra $\d_{g,n}$ introduced  in \cite[\S10]{hain2} plays a key role in \cite{hain2} and our main results.   Let  $\theta$ be the cup product pairing $\Lambda^2H_\Zl\to \Zl(1)$.   The dual $\thetadual$ of $\theta$ can be viewed as an element of $\Lambda^2H_\Zl(-1)$.
Denote the representation $(\Lambda^3H)(-1)/\thetadual\wedge H$ by $\Lambda^3_0H$ and  denote the kernel of $\theta:\Lambda^2H\to \Ql(1)$ by $\Lambda^2_0H$. The representations $\Lambda^3_0H$ and $\Lambda^2_0H$ are irreducible as $\GSp(H)$-modules. 
Define 
$$\d_{g,n}:=\Gr^W_\bullet\u_{g,n}[\ell^r]/(W_{-3} + (\Lambda^2_0H)^\perp) \,\,\text{ and } \,\,\d_{\cC_{g,n}}: = \Gr^W_\bullet\u_{\cC_{g,n}}[\ell^r]/(W_{-3} + (\Lambda^2_0H)^\perp),$$
where $(\Lambda^2_0H)^\perp$ denotes the $\GSp(H)$-complement of the isotypical $\Lambda^2_0H$-component in weight $-2$. 
The fact that $H_1(\p)$ is pure of weight $-1$ implies that by strictness the weight filtration $W_\bullet \p$ agrees with the lower central series of $\p$. Hence, recall the following well known descriptions of $\Gr^W_m\p$ for $m=-1, -2$.
\begin{proposition}\label{low degree p rep}For $g \geq 2$, there are $\GSp(H)$-module isomorphisms
$$\Gr^W_{-1}\p \cong H\,\, \text{ and }\,\, \Gr^W_{-2}\p \cong \Lambda^2_0H.$$
\end{proposition}
The universal curve $\pi:\cC_{g,n/k}[\ell^r]\to \M_{g,n/k}[\ell^r]$ induces a $\GSp(H)$-equivariant graded Lie algebra surjection $\d(\pi):\d_{\cC_{g,n}}\to \d_{g,n}$ whose kernel is given by $\Gr^W_\bullet \p/W_{-3}$: there is the exact sequence of graded Lie algebras
$$0\to \Gr^W_{\bullet}\p/W_{-3} \to \d_{\cC_{g,n}}\overset{\d(\pi)}\to \d_{g,n}\to 0.$$
Hence, Proposition  \ref{comparison} (iii), Proposition \ref{iso in weight -1 and -2}, Proposition \ref{low degree p rep}, and \cite[Thm.~9.11]{hain2} imply that  there are $\GSp(H)$-module isomorphisms  in weight $-1$ and $-2$:
$$\Gr^W_{-1}\d_{g,n} \cong \Lambda^3_0H \oplus \bigoplus_{j=1}^nH_j,\,\, \Gr^W_{-2}\d_{g,n} \cong \bigoplus_{j=1}^n\Lambda^2_0H_j,\,\,\Gr^W_m\d_{g,n} =0 \text{ for } m\leq -3$$
and 
$$\Gr^W_{-1}\d_{\cC_{g,n}} \cong \Lambda^3_0H \oplus \bigoplus_{j=0}^nH_j,\,\, \Gr^W_{-2}\d_{\cC_{g,n}} \cong \bigoplus_{j=0}^n\Lambda^2_0H_j,\,\,\Gr^W_m\d_{\cC_{g,n}} =0 \text{ for } m\leq -3,$$
\textcolor{black}{where $H_j$ is a copy of $H$.}
Furthermore, by \cite[Prop.~10.2]{hain2}, the open immersion $\M_{g,n+1/k}[\ell^r]\hookrightarrow \cC_{g,n/k}[\ell^r]$ induces a $\GSp(H)$-equivariant graded Lie algebra isomorphism $\d_{g,n+1}\overset{\sim}\to \d_{\cC_{g,n}}$ that makes the  diagram 
$$\xymatrix@C=1pc @R=1pc{
		\d_{g, n+1}\ar[r]^{\sim}\ar[dr]&\d_{\cC_{g,n}}\ar[d]^{\d(\pi)}\\
		&\d_{g,n}
}
	$$
	commute, where the surjection $\d_{g, n+1}\to \d_{g,n}$ is induced by the projection $\M_{g,n+1}\to \M_{g,n}$ mapping $[C, x_0, x_1, \ldots, x_n]\mapsto [C, x_1,\ldots,x_n]$. 
	 Each tautological section $s_j$ of $\pi:\cC_{g,n/k}[\ell^r]\to \M_{g,n/k}[\ell^r]$ induces a $\GSp(H)$-equivariant graded Lie algebra section $\d(s_j)$ of $\d(\pi)$ (see \cite[Cor.~10.3]{hain2}). 
	We have the following key result.
	\begin{proposition} [{ \cite[Prop.~10.4 \& 10.8]{hain2}, \cite[Prop.~11.3]{wat1} }] \label{sections of dpi}
	For $g\geq 4$ and $n\geq 1$, the only $\GSp(H)$-equivariant graded Lie algebra sections of $\d(\pi)$ are the sections $\d(s_1),\ldots, \d(s_n)$.	\end{proposition}


\section{Non-abelian cohomology of $\pi_1(\M_{g,n/k}[\ell^r], \etabar)$ }\label{non-ab coho of arith} 
\subsection{Definition} Assume that $g\geq 3$. 
Let $k$ be a finite field with characteristic $p$ and $\ell$ a prime number distinct from $p$.  
Let $r$ be a nonnegative integer. Here set $\G^\arith_{g,n}[\ell^r] := \pi_1(\M_{g,n/k}[\ell^r], \etabar)$. 
 Now, from the exact sequence \begin{equation}\label{6th exact seq for p}1\to \cP\to \cG_{\cC_{g,n}}[\ell^r]\to \cG_{g,n}[\ell^r]\to 1,\end{equation} the conjugation action of  $ \cG_{\cC_{g,n}}[\ell^r]$ on $\cP$ induces the commutative diagram
$$\xymatrix@C=1pc @R=.7pc{
	 		1\ar[r]&\cP\ar@{=}[d]\ar[r]&\cG_{\cC_{g,n}}[\ell^r]\ar[r]\ar[d]&\cG_{g,n}[\ell^r]\ar[d]\ar[r]&1\\
		1\ar[r]&\cP\ar[r]&\Aut(\cP)\ar[r]               &\Out(\cP)\ar[r]                 &1.
	}
	$$
	The right-hand vertical  is the map $\phi:\cG_{g,n}[\ell^r]\to \Out(\cP)$ given in  the introduction. Since $\cG_{\phi}$ is the fiber product $\Aut(\cP)\times_{\Out(\cP), \phi}\cG_{g,n}[\ell^r]$, there is the commutative diagram
$$\xymatrix@C=1pc @R=.7pc{
	 		1\ar[r]&\cP\ar@{=}[d]\ar[r]&\cG_{\cC_{g,n}}[\ell^r]\ar[r]\ar[d]&\cG_{g,n}[\ell^r]\ar@{=}[d]\ar[r]&1\\
		1\ar[r]&\cP\ar[r]&\cG_{\phi}\ar[r]&\cG_{g,n}[\ell^r]\ar[r]&1.	}
	$$	
\textcolor{black}{\begin{lemma}\label{exact seq with A}
For each $\Ql$-algebra $A$, the sequences 
$$
1\to\cP(A)\to \cG_{\cC_{g,n}}[\ell^r](A)\to \cG_{g,n}[\ell^r](A)\to 1
$$
and
$$
1\to\cP^o(A)\to\cG_{g,n+1}[\ell^r](A)\to\cG_{g,n}[\ell^r](A)\to 1.
$$
are exact. 
\end{lemma}
\begin{proof} For each $\Ql$-algebra $A$, the sequence $1\to \cP(A)\to \cG_{\cC_{g,n}}[\ell^r](A)\to \cG_{g,n}[\ell^r](A)$ is exact. Hence it remains to show that $\cG_{\cC_{g,n}}[\ell^r](A)\to \cG_{g,n}[\ell^r](A)$ is surjective. Fix a splitting  of $\cG_{\cC_{g,n}}\to \GSp(H)$ in the category of $\Ql$-groups. This induces a splitting of $\cG_{g,n}[\ell^r]\to \GSp(H)$.  Set $H_A:= H\otimes_\Ql A$. These splittings yield isomorphisms $\cG_{\cC_{g,n}}[\ell^r](A)\cong \U_{\cC_{g,n}}[\ell^r](A)\rtimes \GSp(H_A)$ and $\cG_{g,n}(A)\cong \U_{g,n}[\ell^r](A)\rtimes \GSp(H_A)$, which are compatible with the map $\cG_{\cC_{g,n}}[\ell^r](A)\to \cG_{g,n}[\ell^r](A)$. 
The surjectivity of  the homomorphism $\U_{\cC_{g,n}}[\ell^r]\to \U_{g,n}[\ell^r]$ of prounipotent groups implies that the Lie algebra map $\u_{\cC_{g,n}}[\ell^r]\otimes_\Ql A\to \u_{g,n}[\ell^r]\otimes_\Ql A$ is surjective, and so is the map $\U_{\cC_{g,n}}[\ell^r](A)\to \U_{g,n}[\ell^r](A)$, because the log maps $\U_{\cC_{g,n}}[\ell^r](A)\to \u_{\cC_{g,n}}[\ell^r]\otimes_\Ql A$ and $\U_{g,n}[\ell^r](A)\to  \u_{g,n}[\ell^r]\otimes_\Ql A$ are bijections. Therefore, the map $\cG_{\cC_{g,n}}[\ell^r](A)\to \cG_{g,n}[\ell^r](A)$ is surjective. A similar argument applies to the second sequence. 
\end{proof}}
The exactness of the first sequence in Lemma \ref{exact seq with A} implies that for each $\Ql$-algebra $A$, the map $\cG_{\phi}(A)\to \cG_{g,n}[\ell^r](A)$ is surjective, and the induced homomorphism $\cG_{\cC_{g,n}}[\ell^r]\to \cG_{\phi}$ is an isomorphism.  Pulling back the exact sequence
$$1\to\cP(\Ql)\to \cG_{\cC_{g,n}}[\ell^r](\Ql)\to \cG_{g,n}[\ell^r](\Ql)\to 1$$
along the representation $\tilde{\rho}_\etabar: \G^\arith_{g,n}[\ell^r]\to \cG_{g,n}[\ell^r](\Ql)$, 
we obtain an extension
$$1\to \cP(\Ql)\to \cE_{g,n}\to \G^\arith_{g,n}[\ell^r]\to 1$$
of $\G^\arith_{g,n}[\ell^r]$ by $\cP(\Ql)$. 
Recall that  the weight filtration $W_\bullet\p$ agrees with the lower central series of $\p$. So the filtration $W_\bullet \cP$ induced by the exponential function on $\p$ is also the lower central series of $\cP$.  So each  normal subgroup $W_m\cP$ of $\cP$ is also normal in $\cG_{\cC_{g,n}}[\ell^r]$. For each $N\leq -1$, pushing down the exact sequence (\ref{6th exact seq for p}) along $\cP\to \cP/W_N\cP$, we obtain the exact sequence
\begin{equation}\label{7th exact seq for p}1\to \cP/W_N\cP \to \cG_{\cC_{g,n}}[\ell^r]/W_N\cP \to \cG_{g,n}[\ell^r]\to 1.\end{equation}
Pulling back the exact sequence $(\ref{7th exact seq for p})$ along $\tilde{\rho}_\etabar$ gives an extension
$$1\to (\cP/W_N\cP)(\Ql)\to \cE^N_{g,n}\to \G^\arith_{g,n}[\ell^r]\to 1$$
of $\G^\arith_{g,n}[\ell^r]$ by $(\cP/W_N\cP)(\Ql)$. Define the set of $\cP(\Ql)$-conjugacy classes of \textcolor{black}{continuous} sections of $\cE_{g,n}\to \G^\arith_{g,n}[\ell^r]$ by $H^1_\nab(\G^\arith_{g,n}[\ell^r], \cP(\Ql))$. Similarly, for each $N\leq -1$,  define the set of $(\cP/W_N\cP)(\Ql)$-conjugacy classes of \textcolor{black}{continuous} sections of $\cE^N_{g,n}\to \G^\arith_{g,n}[\ell^r]$ by $H^1_\nab(\G^\arith_{g,n}[\ell^r], \cP/W_N\cP(\Ql))$. 
\textcolor{black}{\begin{variant} Using Proposition \ref{exact seq for punctured universal curve}, we can apply a similar construction to the universal punctured curve  $\pi^o$ and we have $H^1_\nab(\G^\arith_{g,n}[\ell^r], \cP^o(\Ql))$. 
\end{variant}}
\subsection{Non-abelian cohomology scheme of $\cG_{g,n}[\ell^r]$}Denote the Lie algebras of  $\cG_{\cC_{g,n}}[\ell^r]$,  $\U_{\cC_{g,n}}[\ell^r]$, $\cG_{g,n}[\ell^r]$, $\U_{g,n}[\ell^r]$, and $\GSp(H)$,by $\g_{\cC_{g,n}}[\ell^r]$, $\u_{\cC_{g,n}}[\ell^r]$, $\g_{g,n}[\ell^r]$,  $\u_{g,n}[\ell^r]$, and $\r$, respectively.  A spectral sequence produced from the extension $0\to \u_{g,n}[\ell^r]\to\g_{g,n}[\ell^r]\to \r\to0$ implies that for each finite dimensional $\GSp(H)$-module $V$, there are isomorphisms
$$H^j(\g_{g,n}[\ell^r], V)\cong H^0(\r, H^j(\u_{g,n}[\ell^r])\otimes V)\cong \Hom_{\GSp(H)}(H_j(\u_{g,n}[\ell^r]), V). $$
\begin{proposition}\label{condition for existence} For $g\geq 3$ and $n\geq 0$, we have
$$H^1(\g_{g,n}[\ell^r], \Gr^W_m\p) \cong \begin{cases} 
     \oplus_{i=1}^n\Ql & m =-1 \\
    0& m <-1. \\
       \end{cases} $$
\end{proposition}
\begin{proof} Recall that the Lie algebra map $d\beta :\u^\geom_{g,n}[\ell^r]\to \u_{g,n}[\ell^r]$ is surjective and the induced graded Lie algebra map $\Gr(d\beta)$ is an isomorphism in weight $-1$ and $-2$. Since $H_1(\u^\geom_{g,n}[\ell^r])\cong H_1(\u^\geom_{g,n})$  and $H_1(\u^\geom_{g,n})$ is pure of weight $-1$ by \cite[Thm.~9.11]{hain2}, the surjectivity of $d\beta$ implies that $H_1(\u_{g,n}[\ell^r])$ is also pure of weight $-1$. We have $H_1(\u) =\Gr^W_{-1}H_1(\u)$ for $\u \in \{\u^\geom_{g,n}[\ell^r], \u_{g,n}[\ell^r]\}$ and there is a commutative diagram
$$\xymatrix@C=1pc @R=.7pc{
\Gr^W_{-1}\u^\geom_{g,n}[\ell^r]\ar[r]^\sim\ar[d] &\Gr^W_{-1}\u_{g,n}[\ell^r]\ar[d]\\
H_1(\u^\geom_{g,n}[\ell^r])\ar[r] & H_1(\u_{g,n}[\ell^r]),
	 			}
	$$	
where the vertical maps are isomorphisms and the top map is $\Gr_{-1}(d\beta)$. Thus the bottom map in the diagram is an isomorphism. 
Here, we have  isomorphisms
\begin{align*}H^1(\g_{g,n}[\ell^r], \Gr^W_m\p) &\cong \Hom_{\GSp(H)}(H_1(\u_{g,n}[\ell^r]), \Gr^W_m\p)\\
&\cong \Hom_{\GSp(H)}(\Gr^W_mH_1(\u^\geom_{g,n}), \Gr^W_m\p)\\
&\cong \begin{cases}
\oplus_{i=1}^n\Ql & m =-1, \\
    0& m <-1 \\
    \end{cases}
\end{align*}
where the last isomorphism follows from \cite[Thm.~9.11]{hain2}. 
\end{proof}
Proposition \ref{condition for existence} shows that $H^1(\g_{g,n}[\ell^r], \Gr^W_m\p)$ is of finite dimensional for all $m\leq -1$. Therefore, it follows from a result of Hain \cite[Thm.~4.6]{hain4}  that for each $N\leq -1$, there exists an affine $\Ql$-scheme of finite type  $H^1_\nab( \cG_{g,n}[\ell^r], \cP/W_{N}\cP)$ that represents the functor associating to each $\Ql$-algebra $A$ the set of $(\cP/W_{N}\cP)(A)$-conjugacy classes of \textcolor{black}{continuous} sections of $\cG_{\cC_{g,n}}[\ell^r]/W_{N}\cP\otimes_\Ql A\to \cG_{g,n}[\ell^r]\otimes_\Ql A$.
Applying the exact functor $\Gr^W_\bullet$ to the exact sequence (\ref{7th exact seq for p}) gives the exact sequence of the corresponding associated graded Lie algebras 
$$0\to \Gr^W_\bullet\p/W_N\p\to\Gr^W_\bullet\g_{\cC_{g,n}}[\ell^r]/W_N\p\to\Gr^W_\bullet\g_{g,n}[\ell^r]\to 0.$$
For $N \leq -1$, denote by $\mathrm{Sect}_{\GSp(H)}(\Gr^W_\bullet\g_{g,n}[\ell^r], \Gr^W_\bullet\p/W_{N}\p)$ the set of $\GSp(H)$-equivariant graded Lie algebra sections of $\Gr^W_\bullet\g_{\cC_{g,n}}[\ell^r]/W_{N}\p\to\Gr^W_\bullet\g_{g,n}[\ell^r]$.
\begin{proposition}[{\cite[Thm.~4.6 \& Cor.~4.7]{hain4}}]\label{nonabelian cohomology}With  notation as above, 
there are bijections
	$$H^1_\nab( \cG_{g,n}[\ell^r], \cP/W_{N}\cP)(\Ql)\cong \mathrm{Sect}_{\GSp(H)}(\Gr^W_\bullet\g_{g,n}[\ell^r], \Gr^W_\bullet\p/W_{N}\p),$$
	and 
$$ H^1_\nab(\cG_{g,n}[\ell^r], \cP)(\Ql)\cong \varprojlim_{N\leq -1} H^1_\nab( \cG_{g,n}[\ell^r], \cP/W_{N}\cP)(\Ql).$$
\end{proposition}
The following result allows us to use the exact sequence for non-abelian cohomology schemes \cite[\S4.4]{hain4}.  It follows  from the universal property of weighted completion (see \cite[Prop.~15.2]{hain2}).  
\begin{proposition}\label{non-abelian coh iso}With notation as above,  there are bijections
$$H^1_\nab(\G^\arith_{g,n}[\ell^r], \cP/W_N\cP(\Ql))\cong H^1_\nab(\cG_{g,n}[\ell^r], \cP/W_N\cP)(\Ql) \,\,\text{ for all }N\leq -1$$
and 
$$H^1_\nab(\G^\arith_{g,n}[\ell^r], \cP(\Ql))\cong H^1_\nab(\cG_{g,n}[\ell^r], \cP)(\Ql).$$
\end{proposition}
\subsection{Proof of Theorem 1} \textcolor{black}{Suppose that $p$ is a prime number, that $\ell$ is a prime number distinct from $p$, and that $r$ is a nonnegative integer. 
Let $k$ be a finite field with $\Char(k) =p$ that contains all $\ell^r$th roots of unity.  Assume that $g \geq 4$ and $n\geq 1$. }
For each $j=1,\ldots,n$, the tautological section $s_j$ of $\pi:\cC_{g,n/k}[\ell^r]\to \M_{g,n/k}[\ell^r]$ induces a section of  the projection $\cG_{\cC_{g,n}}[\ell^r]\to \cG_{g,n}[\ell^r]$  of the weighted completions and hence a section of $\cG_{\cC_{g,n}}[\ell^r]/W_N\cP\to \cG_{g,n}[\ell^r]$ for $N\leq -1$. Thus each section $s_j$ induces a class in $H^1_\nab(\cG_{g,n}[\ell^r], \cP)(\Ql)$ and $H^1_\nab(\cG_{g,n}[\ell^r], \cP/W_N\cP)(\Ql)$, both of which are denoted by $s_j^\un$.  Furthermore, the homomorphism $\tilde{\rho}_\etabar: \G^\arith_{g,n}[\ell^r]\to \cG_{g,n}[\ell^r](\Ql)$ pulls back each $s^\un_j$ to give a class in $H^1_\nab(\G^\arith_{g,n}[\ell^r], \cP(\Ql))$ and $H^1_\nab(\G^\arith_{g,n}[\ell^r], \cP/W_N\cP(\Ql))$, which are also denoted by $s^\un_j$.  
First, we show that for $g\geq 4$, there is a bijection
$H^1_\nab(\cG_{g,n}[\ell^r], \cP/W_{-3}\cP)(\Ql) = \{s^\un_1,\ldots, s^\un_n\} .$
By Proposition \ref{nonabelian cohomology}, we have 
$$H^1_\nab(\cG_{g,n}[\ell^r], \cP/W_{-3}\cP)(\Ql)\cong \mathrm{Sect}_{\GSp(H)}(\Gr^W_\bullet\g_{g,n}[\ell^r], \Gr^W_\bullet\p/W_{-3}\p).$$
 For each $j =1, \ldots, n$, denote the image of $s^\un_j$ in $\mathrm{Sect}_{\GSp(H)}(\Gr^W_\bullet\g_{g,n}[\ell^r], \Gr^W_\bullet\p/W_{-3}\p)$ by $\Gr^W_\bullet ds^\un_j$.
By Proposition \ref{weight filt}, the functors $V\mapsto\Gr^W_\bullet V $ and $V\mapsto V/W_mV$ are exact. Thus there is the exact sequece
$$0\to\Gr^W_\bullet\p/W_{-3}\to \Gr^W_\bullet\g_{\cC_{g,n}}[\ell^r]/W_{-3}\to  \Gr^W_\bullet\g_{g,n}[\ell^r]/W_{-3}\to 0, $$
where  the map $ \Gr^W_\bullet\g_{\cC_{g,n}}[\ell^r]/W_{-3}\to  \Gr^W_\bullet\g_{g,n}[\ell^r]/W_{-3}$ is denoted by $\Gr^W_{\bullet} d\pi_\ast/W_{-3}$.
 Denote by $\mathrm{Sect}_{\GSp(H)}(\Gr^W_\bullet\g_{g,n}[\ell^r]/W_{-3}, \Gr^W_\bullet\p/W_{-3})$ the set of $\GSp(H)$-equivariant graded Lie algebra sections of $\Gr^W_\bullet d\pi_\ast/W_{-3}$. Note that  $\mathrm{Sect}_{\GSp(H)}(\Gr^W_\bullet\g_{g,n}[\ell^r]/W_{-3}, \Gr^W_\bullet\p/W_{-3})$ is in bijection with the set $\mathrm{Sect}_{\GSp(H)}(\Gr^W_\bullet\g_{g,n}[\ell^r], \Gr^W_\bullet\p/W_{-3}\p).$ Fix $\GSp(H)$-module isomorphisms $\gamma:\Gr^W_{-2}\p \cong \Lambda^2_0H$ and $\alpha:\Gr^W_{-2}\g_{g,n}[\ell^r]\cong \bigoplus_{j=1}^n\Lambda^2_0H_j \oplus (\Lambda^2_0H)^\perp.$
Then in weight $-2$, there is a commutative diagram of $\GSp(H)$-modules
$$\xymatrix@C=1pc @R=.7pc{
0\ar[r] & \Gr^W_{-2}\p \ar[r]\ar[d]^{\gamma}&\Gr^W_{-2}\g_{\cC_{g,n}}[\ell^r]\ar[r]^{\Gr^W_{-2}d\pi_\ast}\ar[d]^{\cong}&\Gr^W_{-2}\g_{g,n}[\ell^r]\ar[r]\ar[d]^{\alpha}&0\\
0\ar[r] &\Lambda^2_0H \ar[r]                          & \Lambda^2_0H \oplus \bigoplus_{j=1}^n\Lambda^2_0H_j\oplus (\Lambda^2_0H)^\perp  \ar[r]  &\bigoplus_{j=1}^n\Lambda^2_0H_j \oplus (\Lambda^2_0H)^\perp \ar[r] &0,
}
$$ 
where the rows are exact and the middle vertical isomorphism is determined by $\gamma$ and $\alpha\circ\Gr^W_{-2}d\pi_\ast$. From the description of the map $\d(\pi)$ in \S\ref{two-step}, it follows that each $\GSp(H)$-equivariant graded Lie algebra section $\Gr^W_\bullet ds/W_{-3}$ of $\Gr^W_{\bullet} d\pi_\ast/W_{-3}$ induces a $\GSp(H)$-equivariant graded Lie algebra section $\d(s)$ of $\d(\pi)$. Note that the restriction of a section of $\Gr^W_\bullet d\pi_\ast/W_{-3}$ to the $(\Lambda^2_0H)^\perp$-component in weight $-2$ is independent of the choice of  a section. Therefore, there is a bijection between $\mathrm{Sect}_{\GSp(H)}(\Gr^W_\bullet\g_{g,n}[\ell^r]/W_{-3}, \Gr^W_\bullet\p/W_{-3})$ and the set of $\GSp(H)$-graded Lie algebra sections of $\d(\pi)$. By Proposition \ref{sections of dpi}, the sections of $\d(\pi)$ are given by $\d(s_1),\ldots, \d(s_n)$.   Thus it follows that $\mathrm{Sect}_{\GSp(H)}(\Gr^W_\bullet\g_{g,n}[\ell^r], \Gr^W_\bullet\p/W_{-3}\p)$ consists of exactly the sections $\Gr^W_\bullet ds^\un_1,\ldots, \Gr^W_\bullet ds^\un_n$. Hence our first claim follows. \\
Next, using the non-abelian exact sequence \cite[Thm.~3]{hain4}, we will show that for each $N\leq -4$, 
$ H^1_\nab(\cG_{g,n}, \cP/W_N\cP)(\Ql) = \{s_1^\un, \ldots, s_n^\un \}. $
Consider the following sequence given by \cite[Thm.~3]{hain4}
 $$H^1_\nab(\cG_{g,n}[\ell^r], \cP/W_{N}\cP)\to H^1_\nab(\cG_{g,n}[\ell^r], \cP/W_{N+1}\cP)\overset{\delta}\to H^2(\g_{g,n}[\ell^r], \Gr^W_{N+1}\p),$$ 
 where $H^1_\nab(\cG_{g,n}[\ell^r], \cP/W_{N}\cP)(\Ql)$ is a principal $H^1(\g_{g,n}[\ell^r], \Gr^W_{N+1}\p)$ set over the set of $\Ql$-rational points $(\delta^{-1}(0))(\Ql)$. For the definition of the map $\delta$, see \cite[\S 14.2]{hain2}. 
By Proposition \ref{condition for existence}, we have $H^1(\g_{g,n}[\ell^r], \Gr^W_{N}\p)=0$ for $N<-1$. Thus we have 
$$H^1_\nab(\cG_{g,n}, \cP/W_{N})(\Ql)\cong \{0\}\times(\delta^{-1}(0))(\Ql) \cong (\delta^{-1}(0))(\Ql).$$
The fact that each $s^\un_j$ lifts from $H^1_\nab(\cG_{g,n}[\ell^r], \cP/W_{N+1}\cP)$ to $ H^1_\nab(\cG_{g,n}[\ell^r], \cP/W_{N}\cP)$ implies that by construction, $\delta(s^\un_j) =0$ for $j =1,\ldots, n$. Therefore, for $N=-3$, we have 
$$(\delta^{-1}(0))(\Ql) = H^1_\nab(\cG_{g,n}[\ell^r], \cP/W_{-3}\cP)(\Ql) = \{s^\un_1,\ldots, s^\un_n\}.$$
Hence, inductively, we have $H^1_\nab(\cG_{g,n}, \cP/W_N\cP)(\Ql)=\{s_1^\un, \ldots, s_n^\un \}$ for all $N\leq -4$ as well. 
Since $H^1_\nab(\cG_{g,n}, \cP)(\Ql) \cong \varprojlim_{N\leq -1}H^1_\nab(\cG_{g,n}, \cP/W_N\cP)(\Ql)$, it follows that
$$H^1_\nab(\cG_{g,n}, \cP)(\Ql) = \{s_1^\un, \ldots, s_n^\un \}.$$
By Proposition \ref{non-abelian coh iso}, we have
$$H^1(\G^\arith_{g,n}[\ell^r], \cP(\Ql)) = \{s_1^\un, \ldots, s_n^\un \}.$$

\qed

\subsection{Proof of Theorem 2}  \textcolor{black}{With the same assumptions as in Theorem 1, consider the exact sequence of weighted completions associated to the universal punctured curve $\pi^o:\M_{g, n+1/k}[\ell^r]\to \M_{g,n/k}[\ell^r]$
$$
1\to \cP^o\to\cG_{g, n+1}[\ell^r]\to \cG_{g,n}[\ell^r]\to 1.
$$ }
\textcolor{black}{Proposition \ref{exact seq for punctured universal curve} implies that  $\cG^\geom_{g, n+1}[\ell^r]$ is isomorphic to the fiber product of $\cG^\geom_{g,n}[\ell^r]$ and $\cG_{g, n+1}[\ell^r]$ over $ \cG_{g,n}[\ell^r]$, and hence a section of the projection $\cG_{g, n+1}[\ell^r]\to \cG_{g,n}[\ell^r]$ induces a section of $\cG^\geom_{g, n+1}[\ell^r]\to \cG^\geom_{g,n}[\ell^r]$. }
Therefore, it will suffice to show that the sequence 
$$
1\to \cP^o\to\cG^\geom_{g, n+1}[\ell^r]\to \cG^\geom_{g,n}[\ell^r]\to 1
$$
does not split. But this directly follows from the comparison between characteristic zero and $p$ in Proposition \ref{comparison} and \cite[Thm.~1]{wat2}. 
\textcolor{black}{Secondly, by Lemma \ref{exact seq with A}} the sequence
$$
1\to \cP^o(\Ql)\to\cG_{g, n+1}[\ell^r](\Ql)\to \cG_{g,n}[\ell^r](\Ql)\to 1
$$
is exact, and by pulling back this sequence along the homomorphism $\tilde\rho_\etabar: \G^\arith_{g,n}[\ell^r]\to \cG_{g,n}[\ell^r](\Ql)$, we obtain an extension 
\begin{equation}\label{nonab open}
1\to \cP^o(\Ql)\to\cE^o_{g,n+1}\to \G^\arith_{g,n}[\ell^r]\to 1
\end{equation}
of $\G^\arith_{g,n}[\ell^r]$ by $\cP^o(\Ql)$. We have the commutative diagram
$$\xymatrix@C=1pc @R=.7pc{
	 		1\ar[r]&\cP^o(\Ql)\ar@{=}[d]\ar[r]&\cE^o_{g,n+1}\ar[r]\ar[d]&\G^\arith_{g,n}[\ell^r]\ar[d]^{\tilde\rho_\etabar}\ar[r]&1\\
		1\ar[r]&\cP^o(\Ql)\ar[r]&\cG_{g, n+1}[\ell^r](\Ql)\ar[r]&\cG_{g,n}[\ell^r](\Ql)\ar[r]&1.	}
	$$	
By the universal property of weighted completion, a continuous section of $\cE^o_{g,n+1}\to \G^\arith_{g,n}[\ell^r]$ induces a section of $\cG_{g, n+1}[\ell^r]\to \cG_{g,n}[\ell^r]$. Thus, the extension (\ref{nonab open}) does not split. 
Consequently, the non-abelian cohomology $H^1_\nab(\G^\arith_{g,n}[\ell^r], \cP^o(\Ql))$ is empty.

\end{document}